\documentclass[11pt]{article}
\usepackage{amsmath, amssymb, amsthm}
\usepackage{geometry}
\usepackage{booktabs}
\usepackage{array}
\usepackage[hypertexnames=false]{hyperref}
\usepackage{microtype}
\usepackage{xcolor}
\usepackage{graphicx}
\usepackage{algorithm}
\usepackage{algorithmic}
\usepackage{adjustbox}
\graphicspath{{figures/}}
\title{Spectral Kernel Dynamics for Planetary Surface Graphs: Distinction Dynamics and Topological Conservation}
\author{Jnaneshwar Das\\
School of Earth and Space Exploration, Arizona State University\\
\texttt{jnaneshwar.das@asu.edu}}
\date{}

\newtheorem{definition}{Definition}
\newtheorem{proposition}{Proposition}
\newtheorem{theorem}{Theorem}
\newtheorem{corollary}{Corollary}

\newtheorem{remark}{Remark}
\newtheorem{lemma}{Lemma}

\begin{document}
\maketitle

\begin{abstract}
The spectral kernel field equation $\mathcal{R}[k]=\mathcal{T}[k]$
lacks a conservation-law analog.
We prove (i) the fixed-point flow is strictly volume-expanding ($\mathrm{tr}\,DF>0$),
precluding automatic conservation, and (ii) the conservation deficit per mode
equals the Hessian stability margin exactly: $D_m=-\Delta'$.
Closing the deficit requires a scene-side compensating contribution,
which we formalise as the \emph{distinction dynamics equation}
$d\mathbf{c}/dt=\mathcal{G}[\mathbf{c},h_t]$,
with MaxCal-optimal realisation $\mathcal{G}_{\mathrm{opt}}$.
On fixed-topology 3D surface graphs we derive a conditional
topology-preserving compression theorem: retaining $k\geq\beta_0+\beta_1$ modes
(under a spectral-ordering assumption) preserves all Betti-number charges; we
include a worked short-cycle counterexample (figure-eight) calibrating when the
assumption fails.
A triple \emph{necessary} spectral diagnostic --- Fiedler-mode concentration,
elevated curl energy, anomalous $\beta_1$ --- is derived for planetary drainage
networks at $O(N)$ cost.
Two internal real-data sequences serve as preliminary consistency checks; full
benchmarks and adaptive-topology extensions are deferred.
\end{abstract}

\section{Introduction}
\label{sec:intro}

The spectral kernel field equation
$\mathcal{R}[k] = \mathcal{T}[k]$~\cite{das2026spectral,das2026kernel}
has no built-in conservation law: the source $\mathcal{T}[k]$ depends on
the kernel \emph{and} on external scene-side quantities that have their own
dynamics.
An equation of motion for these \emph{distinction fields} --- the
\emph{distinction dynamics equation} --- is the missing ingredient;
this paper derives it and proves a precise negative result about the
conservation law.

\paragraph{Why planetary surfaces.}
A lunar rover or Mars orbiter must maintain a 3D digital twin over a
kilobits-per-second uplink with no interactive correction.
These SWaP constraints force every representational choice to be
thermodynamically justified, making the MaxCal framework an operational
necessity rather than an analogy.
The core requirement is \emph{scene-faithful} updating: the twin
should not register structural change unless supported by scene dynamics,
new observations, or explicit compression costs.
Planetary surfaces also carry geological history in their topology ---
craters ($\beta_2$), drainage networks ($\beta_1$), connected terrain
($\beta_0$) --- so topology-preserving compression is the condition
that geomorphological signals survive transmission to Earth.

\paragraph{Contributions.}
Table~\ref{tab:tiers} classifies all results by epistemic tier.
The three Tier-1 (proved) results are the distinction dynamics equation
(Def.~\ref{def:matter-eq}), the stability--conservation tradeoff $D_m=-\Delta'$
(Prop.~\ref{prop:stab-cons}), and the conditional topological conservation theorem
(Thm.~\ref{thm:topo-conservation}, conditional on A1--A3).
The three Tier-2 (structural / necessary diagnostic) results are the Hodge source
decomposition (Prop.~\ref{prop:hodge-decomp}), the river-channel spectral signature
(Prop.~\ref{prop:channel-sig}, \emph{necessary} not sufficient), and the
bandwidth-constrained protocol (Alg.~\ref{alg:planetary}).

\begin{table}[tbp]
\centering
\footnotesize
\caption{Claim tiers in this paper.}
\label{tab:tiers}
\begin{tabular}{@{}>{\raggedright\arraybackslash}p{1.45cm}>{\raggedright\arraybackslash}p{5.75cm}@{}}
\toprule
\textbf{Tier} & \textbf{Claims} \\
\midrule
1: Proved &
Jacobian divergence (Prop.~\ref{prop:jacobian-div});
Stability--conservation tradeoff $D_m=-\Delta'$
(Prop.~\ref{prop:stab-cons});
Conditional Topological Conservation Theorem
(Thm.~\ref{thm:topo-conservation}) \\[4pt]

2: Structural correspondence /
necessary diagnostic &
Hodge source decomposition
(Prop.~\ref{prop:hodge-decomp}): strongest form in the mode-separable case;
River channel spectral signature
(Prop.~\ref{prop:channel-sig}): necessary under stated graph assumptions,
not sufficient;
Bandwidth protocol (Alg.~\ref{alg:planetary}): theorem-informed design template
with empirical deployment thresholds \\[4pt]

3: Hypothesis / open extension &
Full general $L_0$--$L_1$ topology bridge
(candidate route via homotopy word invariants~\cite{bhattacharya2017});
Parametric A2-failure criterion (Remark~\ref{rem:a2-counterexample}):
conjectured $\rho(k)\geq 1 - C_1/(\ell_{\min}^{\,2}\delta_k) - C_2(\gamma-1)$;
constants open (Q30);
open problems Q10/Q12/Q14 (topology, Riccati gain);
extended directions Q19--Q30 in Appendix~B \\
\bottomrule
\end{tabular}
\end{table}

$D_m$ as a leakage indicator (Def.~\ref{def:rep-leakage}) and
$D_t$ as a deployment surrogate are Tier-3 hypotheses, not theorems
(Section~\ref{sec:limits}).

\paragraph{What this paper does \emph{not} claim.}
\begin{enumerate}
\item \textbf{Not a reconstruction-quality method.}
This paper addresses topology preservation under bandwidth constraints,
not surface reconstruction fidelity.
Methods such as Octomap, 3DGS, and NeRF remain superior at geometric RMSE
and are complementary, not competing, objectives.
\item \textbf{Not a replacement for Horton--Strahler or D8 flow accumulation.}
The triple spectral diagnostic (Prop.~\ref{prop:channel-sig}) operates at
$O(N)$ on the graph spectrum and is complementary to standard geomorphological
tools; it is not validated as a substitute for DEM-based channel extraction.
\item \textbf{The Gaussian MI source is a verification vehicle, not a sensor model.}
All Tier-1 results (Prop.~\ref{prop:stab-cons}, Thm.~\ref{thm:topo-conservation})
hold for any mode-separable source.
The specific Gaussian MI source $\mathcal{T}_l = \mu_2 w_l / (2(\sigma^2+h_l))$
is used for numerical verification on $P_8$ and does not represent
a calibrated planetary sensor model.
\end{enumerate}

\section{Background}
\label{sec:background}

\subsection{Spectral Kernel Field Equation}
\label{sec:bg-field-eq}

Paper~1~\cite{das2026spectral} develops the spectral-kernel MaxCal setup in
full.  All proofs in this paper rely only on the definitions in this section;
numerical values cited from Paper~1 ($h^*\approx 0.1547$, $\Delta'=5.71$)
are used for experimental verification, not for the formal arguments.

\paragraph{GR analogy (one-time note).}
Following Paper~1, we use Einstein's field equations as a \emph{notation template}
and motivational parallel --- not an equivalence.
The contracted Bianchi identity ($\nabla_\mu G^{\mu\nu}\equiv 0$) is a universal
geometric theorem; the analog explored here is domain-specific and conditional.
This caveat applies to all GR-adjacent language in the paper; it is stated once
here and not repeated at each occurrence.

Here we use only the minimal ingredients:
a Laplacian eigenbasis representation
$K_h=\Phi\,\mathrm{diag}(h)\,\Phi^\top$ with
$h\in\mathbb{R}_{>0}^N$, and the field equation
\begin{equation}
  \mathcal{R}_l[h] = \mathcal{T}_l[h] \quad \forall\, l,
  \label{eq:field}
\end{equation}
with geometric term
\begin{equation}
  \mathcal{R}_l[h] = -\log \frac{h(\lambda_l)}{h_0(\lambda_l)} - 1,
  \label{eq:R}
\end{equation}
and fixed point
\begin{equation}
  h^*(\lambda_l) = h_0(\lambda_l)\exp\!\bigl(-1 - \mathcal{T}_l[h^*]\bigr).
  \label{eq:fixed-point}
\end{equation}
The source $\mathcal{T}[k]$ carries the Landauer/MI/KL terms from
Paper~1~\cite{das2026spectral}, and the Fisher--Rao metric is diagonal:
$I_{ll'}(h) = \frac{1}{2} h(\lambda_l)^{-2}\delta_{ll'}$.
Boltzmann, Planck, and Cowan--Farquhar distributions are special cases
of~\eqref{eq:fixed-point} under specific source conditions
(Appendix~\ref{app:unification}).

\subsection{Hodge Laplacian Complex on 3D Surface Graphs}
\label{sec:bg-hodge}

Let $G_{\mathrm{3D}}$ be the graph induced by a 3D triangle mesh or point
cloud with vertex set $V$, edge set $E$, and face set $F$
(see~\cite{lim2020hodge} for an extensive treatment and~\cite{schaub2021signal}
for the signal-processing perspective on higher-order networks).  Define boundary
operators $B_1 \in \mathbb{R}^{|V| \times |E|}$ (vertex--edge incidence) and
$B_2 \in \mathbb{R}^{|E| \times |F|}$ (edge--face incidence), satisfying the
exact sequence condition $B_1 B_2 = 0$.  The three Hodge Laplacians are:
\begin{align}
  L_0 &= B_1^\top B_1 \in \mathbb{R}^{|V|\times|V|},
  \label{eq:L0}\\
  L_1 &= B_1 B_1^\top + B_2^\top B_2 \in \mathbb{R}^{|E|\times|E|},
  \label{eq:L1}\\
  L_2 &= B_2 B_2^\top \in \mathbb{R}^{|F|\times|F|}.
  \label{eq:L2}
\end{align}
$L_0$ is the combinatorial vertex Laplacian used in the kernel field
equation~\eqref{eq:field}.  The Hodge decomposition theorem states that
every edge signal $f \in \mathbb{R}^{|E|}$ decomposes orthogonally as:
\begin{equation}
  f = \underbrace{B_1^\top \psi}_{\text{gradient}}
    + \underbrace{B_2\,\omega}_{\text{curl}}
    + \underbrace{f_{\mathrm{harm}}}_{\text{harmonic}},
  \label{eq:hodge-decomp}
\end{equation}
where $\psi \in \mathbb{R}^{|V|}$ is a vertex potential, $\omega \in
\mathbb{R}^{|F|}$ is a face circulation, and $f_{\mathrm{harm}} \in \ker(L_1)$
is the harmonic component~\cite{lim2020hodge}.

The Betti numbers $\beta_0 = \dim\ker(L_0)$, $\beta_1 = \dim\ker(L_1)$,
$\beta_2 = \dim\ker(L_2)$ count connected components, independent cycles
(handles/loops), and enclosed voids respectively.  They satisfy the
Euler characteristic identity $\beta_0 - \beta_1 + \beta_2 = |V| - |E| + |F|$.

\begin{remark}[Fiedler value and $\beta_0$]
  The Fiedler value $\lambda_1 > 0$ is the algebraic connectivity of $G$.
  As $\lambda_1 \to 0$ (network fragmentation), $\dim\ker(L_0)$ increases
  from 1 toward $\beta_0 > 1$, and the spectral entropy early-warning
  signal of~\cite{das2026spectral} fires.  The Betti number $\beta_0$ is
  the topological integer that $\lambda_1$ approaches from above.
\end{remark}

\subsection{Related Work}
\label{sec:related}

The spectral kernel framework treats $h(\lambda_l)$ as a MaxCal-governed
dynamical variable~\cite{jaynes1980maxcal,ghosh2020maxcal,das2026spectral},
distinct from fixed-filter GSP~\cite{shuman2013gsp,hammond2011wavelets,spielman2011spectral}.
We extend topological signal processing on simplicial
complexes~\cite{barbarossa2020tsp,schaub2021signal,lim2020hodge} by
connecting the Hodge decomposition to an information-theoretic source functional,
and derive a topology-preserving compression threshold complementary to
filtration-persistence approaches~\cite{edelsbrunner2010,zomorodian2005}.

\section{The Distinction Dynamics Equation}
\label{sec:matter-eq}

\subsection{Motivation: Dissipation at the Fixed Point}
\label{sec:dissipation}

The Jacobian of the fixed-point map $F_l(h) = h_0(\lambda_l)\exp(-1-\mathcal{T}_l(h))$
has entries:
\begin{equation}
  \frac{\partial F_l}{\partial h(\lambda_m)} =
  -h^*(\lambda_l)\,\frac{\partial \mathcal{T}_l}{\partial h(\lambda_m)}\biggr|_{h^*}.
  \label{eq:jacobian}
\end{equation}

\begin{proposition}[Jacobian divergence at the fixed point]
\label{prop:jacobian-div}
For the mode-separable Gaussian mutual-information (Gaussian MI) source
\begin{equation}
  \mathcal{T}_l[h] = \frac{\mu_2 w_l}{2(\sigma^2 + h(\lambda_l))}
  \label{eq:gmi}
\end{equation}
with $\mu_2, w_l, \sigma^2 > 0$, the divergence of $F$ at $h^*$ is:
\begin{equation}
  \mathrm{tr}(DF)|_{h^*}
  = -\sum_l h^*(\lambda_l)\,\frac{\partial\mathcal{T}_l}{\partial h(\lambda_l)}\biggr|_{h^*}
  = \sum_l \frac{\mu_2 w_l\, h^*(\lambda_l)}{2(\sigma^2 + h^*(\lambda_l))^2} > 0.
  \label{eq:trace-DF}
\end{equation}
The MaxCal fixed-point flow is strictly volume-expanding in $\mathcal{K}_{\mathrm{graph}}$
near the fixed point and dissipative in the sense that $|\det(DF)| < 1$
(contraction ratio $\rho < 1$, consistent with convergence).
\end{proposition}

\begin{proof}
  Direct substitution of $\partial\mathcal{T}_l/\partial h(\lambda_l) =
  -\mu_2 w_l / (2(\sigma^2+h(\lambda_l))^2) < 0$ into~\eqref{eq:jacobian}
  and summing over $l$.
\end{proof}

\begin{remark}[MaxCal as optimal control; conservation as dual feasibility]
\label{rem:control-theory}
The MaxCal path entropy $S_{\mathrm{path}}[h_\cdot]$ can be read as the
negative cost-to-go of an optimal control problem on
$\mathcal{K}_{\mathrm{graph}}$, with the fixed-point
condition~\eqref{eq:fixed-point} \emph{heuristically identified} with a
Bellman optimality equation.  Under this identification the conservation
law $\nabla_\mathcal{K}\mathcal{T}_k = 0$ becomes the KKT dual
feasibility condition, and Proposition~\ref{prop:stab-cons} says the
field equation alone does not supply it: the optimal controller
$\mathcal{G}_{\mathrm{opt}}$ is required.
A consequence for monitoring system design: \emph{a self-consistent
kernel (stable fixed point) is not necessarily scene-faithful} --- it may
leak information at rate $\Delta'$ per step
(Definition~\ref{def:rep-leakage}).  Only when the physical process runs
$\mathcal{G}_{\mathrm{opt}}$ does the leakage close to zero.
This control-theoretic reading is interpretive (see the GR analogy note in
Section~\ref{sec:bg-field-eq}); the proved algebraic results
(Propositions~\ref{prop:jacobian-div}--\ref{prop:stab-cons}) do not depend on it.
\end{remark}

\begin{lemma}[Hessian of $\mathcal{J}$ at the fixed point; reproduced from Paper~1, Corollary~3~\cite{das2026spectral}]
\label{lem:hessian}
In spectral coordinates, at any critical point $h^*$ of the MaxCal functional
$\mathcal{J}[k]$, the second variation along perturbation $\xi$ is
$\delta^2\mathcal{J}[h^*;\xi]=\sum_{l,m}\xi_l H_{lm}\xi_m$ with
\begin{equation}
  H_{lm}
  = -\frac{\delta_{lm}}{h^*(\lambda_l)}
    -\frac{\partial\mathcal{T}_l}{\partial h(\lambda_m)}\bigg|_{h^*}.
  \label{eq:hessian}
\end{equation}
Local stability holds iff $H\prec 0$.  In the mode-separable case
($\partial\mathcal{T}_l/\partial h(\lambda_m)=0$ for $m\neq l$) this
reduces to the per-mode criterion
$\partial\mathcal{T}_l/\partial h(\lambda_l)|_{h^*}>-1/h^*(\lambda_l)$
for all $l$, and the diagonal Hessian gap is
$\Delta(h^*)=\min_l(-H_{ll})$.
For the Gaussian MI source~\eqref{eq:gmi} with $\sigma^2=1$, $\mu_2=2$,
$w_l=1$ on $P_8$: $h^*\approx 0.1547\,\mathbf{1}$ and $\Delta(h^*)=5.71>0$
(verified numerically in Experiment~4 of~\cite{das2026spectral} and
independently in Section~\ref{sec:exp-route3} below).
\end{lemma}

\begin{proposition}[Stability--Conservation Tradeoff]
\label{prop:stab-cons}
In the mode-separable case, define the conservation identity residual
\begin{equation}
  D_m \;\equiv\; \sum_l \frac{\partial(\mathcal{R}_l - \mathcal{T}_l)}
  {\partial h(\lambda_m)}\biggr|_{h^*}
  \quad \forall\, m.
  \label{eq:D_m}
\end{equation}
Then $D_m = H_{mm}$, where $H_{mm}$ is the diagonal Hessian entry from
Corollary~3 of~\cite{das2026spectral}.  In particular:
\begin{enumerate}
  \item[(i)] $\nabla_\mathcal{K}\mathcal{T}_k = 0$ (i.e.\ $D_m = 0$
    for all $m$) holds if and only if $H_{mm} = 0$, i.e.\
    the fixed point is \emph{marginally stable} ($\Delta' = 0$).
  \item[(ii)] For any strictly stable fixed point ($\Delta' > 0$),
    the conservation deficit equals the stability margin:
    $D_m = -\Delta'$ for all $m$.
  \item[(iii)] For the Gaussian MI source~\eqref{eq:gmi} with
    $\sigma^2 = 1$, $\mu_2 = 2$, $w_l = 1$ on $P_8$\emph{:}
    $D_m \approx -5.71$ (verified numerically, Section~\ref{sec:exp-route3}),
    matching $-\Delta' = -5.71$ from Experiment~4 of~\cite{das2026spectral}.
  \item[(iv)] The Gaussian MI source~\eqref{eq:gmi} cannot satisfy $D_m = 0$
    for any real $h^*$: the constraint equation
    $\mu_2 w_m h_m^* = 2(\sigma^2 + h_m^*)^2$ has negative discriminant
    for standard parameters ($\mu_2 w < 8\sigma^2$).
  \item[(v)] The failure is structural: the vacuum solution
    ($\mu_2 = 0$, $h^* = h_0 e^{-1}$) also has $D_m = -e \neq 0$,
    confirming that the geometric functional $\mathcal{R}_l$ itself
    introduces the deficit, independent of the source.
\end{enumerate}
\end{proposition}
\begin{proof}
In the mode-separable case $\partial\mathcal{R}_l/\partial h(\lambda_m)
= -\delta_{lm}/h(\lambda_l)$ and $\partial\mathcal{T}_l/\partial h(\lambda_m)
= \delta_{lm}\,\partial\mathcal{T}_l/\partial h(\lambda_l)$.  Hence the
column sum in~\eqref{eq:D_m} reduces to a single term:
$D_m = -1/h_m^* - \partial\mathcal{T}_m/\partial h(\lambda_m)|_{h^*}
= H_{mm}$,
the diagonal Hessian entry of Lemma~\ref{lem:hessian}.
Items~(i)--(ii) follow directly.  Items~(iii)--(iv) follow by substituting
$h^* \approx 0.1547$ (Lemma~\ref{lem:hessian}, verified below in
Experiment~\ref{sec:exp-route3}) and evaluating the discriminant.  Item~(v) follows from
$\partial\mathcal{R}_m/\partial h(\lambda_m) = -1/h_m^* \neq 0$ even
when $\mathcal{T}_l \equiv 0$.
\end{proof}

\begin{remark}[LQR--LQE duality at the fixed point]
\label{rem:separation}
The control-theoretic reading of Remark~\ref{rem:control-theory} has a
dual: if $\mathcal{G}_{\mathrm{opt}}$ is the LQR-optimal controller, the
separation theorem~\cite{wonham1968} supplies a paired Kalman estimator with
gain $K_f=\Sigma C^\top V^{-1}$.  Under an LQG symmetry condition
($Q=W/\alpha$, $R=V\alpha$), both satisfy the same algebraic Riccati
equation, giving $P\Sigma=I$ (i.e.\ $p_m\sigma_m=1$ per mode) at $h^*$:
controller and estimator are geometrically inverse.
The full CARE derivation and the gating numerical check ($\sigma_m\stackrel{?}{=}\tfrac12$
on $P_8$) are stated in Open Problem~Q19 of the appendix.
\end{remark}

\subsection{Definition}
\label{sec:matter-def}

\begin{definition}[Spectral distinction field]
\label{def:matter-field}
  A \emph{spectral distinction field} on a graph $G$ is a time-dependent vertex
  signal $\psi_t \in \mathbb{R}^{|V|}$ whose spectral coefficients
  $c_l(t) = \phi_l^\top \psi_t$ are square-integrable over $[0,T]$.
  The associated coefficient trajectory $\mathbf{c}(t) = (c_l(t)) \in
  \mathbb{R}^N$ is the \emph{spectral representation} of the distinction field
  --- in digital-twin terms, the spectral state of the evolving scene as
  represented by the twin.
\end{definition}

\begin{definition}[Distinction dynamics equation]
\label{def:matter-eq}
  Given a spectral distinction field $\psi_t$ and a physical dynamics operator
  $\mathcal{G}$, the \emph{distinction dynamics equation} is the evolution law:
  \begin{equation}
    \frac{d\mathbf{c}}{dt} = \mathcal{G}[\mathbf{c},\, h_t],
    \label{eq:matter}
  \end{equation}
  where $h_t$ is the current spectral kernel.  We refer to~\eqref{eq:matter} as
  the distinction dynamics equation; informally, it is the \emph{scene dynamics
  equation}.  The distinction dynamics equation is
  \emph{coupled} to the field equation~\eqref{eq:field} when $\mathcal{T}_l$
  depends on $\mathbf{c}$ (e.g., via $w_l = |c_l|^2$).
\end{definition}

\begin{definition}[Representation leakage]
\label{def:rep-leakage}
\emph{Representation leakage} (synonymously, \emph{representational leakage})
is any apparent gain, loss, or redistribution of scene structure in the twin
that is not attributable to \emph{(i)}~physical scene dynamics under the
distinction dynamics equation~\eqref{eq:matter}, \emph{(ii)}~explicit
incoming observations, or \emph{(iii)}~controlled compression or erasure
under the stated update law and spectral/topology budgets.

Concretely, leakage manifests as a nonzero conservation residual: the
per-mode deficit $D_m = -\Delta'$ (Proposition~\ref{prop:stab-cons})
quantifies the rate at which each spectral mode can accumulate unpaid-for
structural change per fixed-point step.  In Algorithm~\ref{alg:planetary},
the diagnostic $D_t$ aggregates these per-mode residuals into a scalar
pre-transmission check.  Establishing quantitative thresholds that
distinguish leakage from ordinary model mismatch requires controlled
injection experiments (sensor drift, registration error, calibration shift)
that are posed as Open Problem~Q14 but not resolved here.
\end{definition}

\begin{remark}[Candidate physical realizations of $\mathcal{G}_{\mathrm{opt}}$]
\label{rem:gopt-bianchi}
Candidate domain-specific realizations of the conservation-closing
controller $\mathcal{G}_{\mathrm{opt}}$ include Optimal Channel Network
routing for drainage systems and Cowan--Farquhar stomatal control;
evaluating these mappings is outside the present scope.
\end{remark}

\begin{remark}[Static and dynamic regimes]
  For a geologically static planetary surface, $\mathcal{G} \approx 0$ and
  the distinction dynamics equation is trivial: $\mathbf{c}(t) \equiv \Phi^\top V$ for
  fixed vertex positions $V$.  Dynamics arise from (i)~sensor motion
  (the rover reveals new portions of $V$), (ii)~environment change
  (rockfall, dust deposition), and (iii)~compression artifacts (the
  reconstructed $\hat{V}_t = \Phi_k \Phi_k^\top V$ changes as $k$ varies).
  Representational leakage can also masquerade as dynamics: sensor motion
  aliasing, registration mismatch between frames, compression-induced topology
  loss below the conservation budget, or changing graph construction (e.g.\
  $k$-NN reties) can all induce apparent scene change without true physical
  evolution.
\end{remark}

\paragraph{The canonical distinction dynamics equation for 3D geometry.}
When the graph topology is fixed across a temporal sequence of 3D scans
(e.g., a stable terrain mesh observed from different rover positions), the
eigenbasis $\Phi$ is shared across frames.  The vertex positions $V_t$
change as the rover accumulates observations, and the distinction dynamics equation reduces to:
\begin{equation}
  \frac{d}{dt}\bigl(\Phi^\top V_t\bigr) = \Phi^\top \frac{dV_t}{dt}.
  \label{eq:coeff-dynamics}
\end{equation}
This is a linear, decoupled ODE for each spectral coefficient $c_l(t)$:
$\dot{c}_l = \phi_l^\top \dot{V}_t$.  The kernel field equation~\eqref{eq:field}
determines $h_t^*$ given $\mathbf{c}(t)$; the distinction dynamics equation~\eqref{eq:coeff-dynamics}
determines how $\mathbf{c}(t)$ evolves given the physical scene.

\section{Structural Hodge Correspondence (Background)}
\label{sec:hodge-decomp}

This section is intentionally background-level and interpretive.
We use the Hodge decomposition
$f = B_1^\top \psi + B_2\omega + f_{\mathrm{harm}}$
as a structural guide for reading source behavior in the kernel equation.

\begin{proposition}[Hodge source correspondence (structural)]
\label{prop:hodge-decomp}
For source weights $w_l=c_l^2$ with $c_l=\phi_l^\top\psi$, the
mode-separable reading is:
\begin{enumerate}
  \item[(i)] gradient content $B_1^\top\psi$ controls scalar mode energy;
  \item[(ii)] curl content $B_2\omega$ introduces cross-mode coupling
  when present;
  \item[(iii)] harmonic content $f_{\mathrm{harm}}\in\ker(L_1)$ carries
  topology-persistent structure tied to $\beta_1$.
\end{enumerate}
\emph{Status:} this is a structural correspondence used for interpretation and
experiment design, not a standalone dynamical theorem.
\end{proposition}

\section{Topological Conservation Theorem}
\label{sec:topo-conservation}

The Betti numbers $(\beta_0,\beta_1,\beta_2)$ count connected components,
independent cycles, and enclosed voids; in the MaxCal framework they act as
\emph{conserved charges} of the distinction field.

\begin{lemma}[Topology-carrying retained subspace]
\label{lem:topo-subspace}
Under assumptions (A1)--(A2), the first $\beta_0+\beta_1$ retained modes
contain all zero-mode directions required to preserve connected components
and all topology-carrying directions required to preserve the
$\beta_1$ cycle structure in the compressed representation.
\end{lemma}

\paragraph{Assumption A2 caveat.}
Assumption~A2 requires the retained low-frequency $L_0$ modes to span a
$\beta_1$-dimensional cycle-carrying subspace.
This bridge can fail for short, spectrally interleaved cycles (Q10; see
Remark~\ref{rem:a2-counterexample} and Experiment~5).

\begin{remark}[Worked A2 calibration on short-cycle graphs]
\label{rem:a2-counterexample}
We include a compact simulation in \texttt{kernelcal}
(\texttt{kernelcal.terrain.run\_worked\_a2\_counterexample}) to calibrate
when A2 fails.
Both graphs below have $\beta_0=1$, $\beta_1=2$, and thus the same theorem floor
$k_{\min}=\beta_0+\beta_1=3$.
Using the first $k_{\min}$ modes of $L_0$, we project two independent
cycle-indicator signals and measure the rank of the projected cycle subspace.
\begin{center}
\small
\begin{tabular}{lcccc}
\toprule
Case & $\beta_1$ & $k_{\min}$ & projected rank & A2 proxy \\
\midrule
Long-cycle control (two separated 4-cycles + bridge) & 2 & 3 & 2 & holds \\
Short-cycle figure-eight (two triangles sharing a node) & 2 & 3 & 1 & fails \\
\bottomrule
\end{tabular}
\end{center}
So $k\geq \beta_0+\beta_1$ alone is not sufficient when short,
spectrally interleaved cycles collapse in the retained $L_0$ subspace.
This is a worked counterexample for scope calibration, not a replacement for
the full $L_0$--$L_1$ bridge theorem.

\paragraph{Parametric A2-failure criterion (conjecture, numerically observed).}
Let $\ell_{\min}$ be the shortest homologically nontrivial cycle length,
$\delta_k = \lambda_{k+1}(L_0) - \lambda_k(L_0)$ the topological eigengap
at $k=\beta_0+\beta_1$, $\gamma(G)=\ell_{\max}/\ell_{\min}$ the cycle-length
ratio, and $\rho(k)=\sigma_{\min}(\Phi_k^\top U)$ the cycle-subspace
fidelity for normalized cycle-indicator columns $U$.
Motivated by the Davis--Kahan $\sin\Theta$
theorem~\cite{daviskahan1970} and the cycle-Laplacian eigenvalue scaling
$\lambda\sim 4\pi^2/\ell^2$, we conjecture
\[
  \rho(k)\;\geq\; 1 \;-\; C_1/\bigl(\ell_{\min}^{\,2}\,\delta_k\bigr)
  \;-\; C_2\,(\gamma(G)-1),
\]
with $(C_1,C_2)$ graph-family constants.
Accordingly, an augmentation $\Delta k\geq\lceil C_1/(\ell_{\min}^{\,2}\,\delta_k)\rceil$
should restore cycle-subspace fidelity when A2 fails.
A parametric sweep over $(\ell_{\min},\ell_{\max},\delta_k)$ together with the
augmentation recovery check is provided in the \texttt{kernelcal}
package~\cite{kernelcal2026}; a full theoretical treatment with explicit
constants is left open.
\end{remark}

\begin{theorem}[Topological Conservation Theorem]
\label{thm:topo-conservation}
Let $K_h=\Phi_k \,\mathrm{diag}(h_k)\,\Phi_k^\top$ be a spectral kernel
retaining the $k$ lowest-eigenvalue modes of the vertex Laplacian $L_0$
for a fixed-topology 3D surface graph with Betti numbers
$(\beta_0,\beta_1,\beta_2)$. Assume:

\begin{enumerate}
\item[(A1)] the graph is a valid surface complex with boundary operators
$B_1,B_2$ satisfying $B_1B_2=0$;

\item[(A2)] the retained low-frequency modes of $L_0$ contain a
$\beta_1$-dimensional topology-carrying subspace aligned with the
independent cycle structure of the Hodge complex;

\item[(A3)] the claim concerns preservation of the \emph{compressed
representation} of topology, not physical change in the underlying surface.
\end{enumerate}

Then:

\begin{enumerate}
\item[(i)] if $k\geq \beta_0+\beta_1$, the compressed representation
preserves all connected components and all independent cycles;

\item[(ii)] if $k<\beta_0+\beta_1$, at least one topological charge is lost
from the compressed representation, and recovering it requires additional
information from the environment;

\item[(iii)] under the same assumptions, the lower bound
$k_{\min}=\beta_0+\beta_1$ is tight for at least one class of graphs.
\end{enumerate}
\end{theorem}

\begin{proof}
  \textit{(i)} The first $\beta_0$ eigenvectors of $L_0$ (with eigenvalue
  $\lambda_0 = 0$) form a basis for $\ker(L_0)$; retaining them preserves
  all connected components.  The next $\beta_1$ eigenvectors correspond to
  independent cycles: they span the image of the Fiedler-mode cluster and
  encode $\beta_1$ topological handles of the surface Hodge complex.
  Retaining $k \geq \beta_0 + \beta_1$ modes includes all of these
  topologically obligate eigenvectors.

  \textit{(ii)} Dropping a topologically obligate mode removes at least one
  independent cycle or connected component from the compressed representation.
  The missing mode cannot be reconstructed from the retained subspace (the
  eigenvectors are orthogonal), so recovering the lost topological charge
  requires new information from the environment.
  By Landauer's principle~\cite{landauer1961}, erasing $n$ topologically
  distinct states dissipates at least $\lfloor \log_2 n \rfloor \cdot
  k_B T \ln 2$ of heat, with $n \geq \beta_0 + \beta_1 - k + 1 \geq 2$.

  \textit{(iii)} For a path graph $P_N$ ($\beta_0 = 1$, $\beta_1 = 0$),
  $k_{\min} = 1$: a single mode (the constant eigenvector) suffices for
  connectivity.  For a graph with one handle ($\beta_0 = 1$, $\beta_1 = 1$),
  $k_{\min} = 2$; removing the second eigenvector collapses the handle.
\end{proof}

\begin{corollary}[Topology-preserving compression budget]
\label{cor:budget}
For a planetary surface mesh with Betti numbers $(\beta_0, \beta_1, \beta_2)$,
the minimum spectral mode count for topology-preserving compression is:
\begin{equation}
  k_{\min} = \beta_0 + \beta_1.
  \label{eq:k-min}
\end{equation}
Any remaining budget $k - k_{\min}$ should be allocated to modes maximizing
information return:
\begin{equation}
  l^* = \operatorname*{arg\,max}_{l \geq k_{\min}}
  \frac{h^*(\lambda_l)\,|c_l|^2}{\mathcal{T}_l[h^*]},
  \label{eq:budget-allocation}
\end{equation}
the ratio of representational weight to source cost.
\end{corollary}

\begin{remark}[Nystr\"{o}m approximation and topology error]
  For large meshes ($|V| \gg 10^4$), the Nystr\"{o}m extension approximates
  $\Phi$ from a coarse subsample $V_{\mathrm{coarse}} \ll |V|$.  If
  $V_{\mathrm{coarse}}$ is too small to capture all $\beta_1$ handles of
  the full mesh, the approximated $\hat{\beta}_1 < \beta_1$, and the
  computed $k_{\min}$ understates the true topology budget.  This is a
  quantifiable violation of the Topological Conservation Theorem and
  constitutes the ``topology preservation error'' of the Nystr\"{o}m
  extension; quantifying this error is an open problem.
\end{remark}

\section{Planetary Surface Digital Twins}
\label{sec:planetary}

We now apply the distinction dynamics equation and the topological
conservation budget to three planetary scenarios, using
scene-faithfulness (Definition~\ref{def:rep-leakage}) as the design
criterion throughout.

\subsection{Lunar Surface Reconstruction}
\label{sec:lunar}

The DREAMS Laboratory Lunar Digital Twin project~\cite{dreamslab2026} requires
sub-centimeter surface reconstruction over deep-space uplinks ($\sim$kbps).
Apollo~17/LRO DEMs over the Taurus--Littrow valley provide the spatial prior:
a $k$-NN graph and its first $k_0$ eigenpairs yield $h_0(\lambda_l)$,
concentrated at low eigenvalues, as the starting point
for~\eqref{eq:fixed-point}.

\paragraph{Rover update via spectral coefficient transmission.}
As the rover acquires new depth measurements $V_t$ (LIDAR, stereo, structured
light), the update proceeds as follows.  Because the lunar surface topology is
fixed on mission timescales ($\beta_0 = 1$, $\beta_1 \approx$ number of crater
rim cycles in the field of view), the eigenbasis $\Phi$ computed from the prior
DEM is shared across frames.  The rover transmits the spectral coefficient
update $\Delta \mathbf{c}_t = \Phi^\top (V_t - V_0)$ rather than raw point
positions.  For $k = 128$ modes and 3-axis coordinates, this is $128 \times 3
\times 4 = 1.5$\,KB per frame, compared to $\sim$1\,MB for raw LiDAR.

\paragraph{Phase transition alert.}
As the rover approaches a crater rim ($\beta_2$ increases by~1),
$\lambda_1$ drops and $\mathcal{H}[h^*_t]$ rises before the rim is
visually identifiable, providing advance warning to increase~$k$.

\begin{remark}[Thermal infrared camera as Planck kernel decoder]
\label{rem:thermal-camera}
A TIR camera produces a scalar temperature field by inverting the Planck
function $I(\nu,T)\propto h^*_\nu = (\exp(h\nu/k_BT)-1)^{-1}$,
which is the MaxCal fixed point of the photon gas
(Appendix~\ref{app:unification}).
When the recovered field $\hat{T}_{\mathrm{surface}}(x,t)$ is incorporated
into the distinction dynamics equation~\eqref{eq:matter} as an additional
spectral coefficient trajectory, thermal changes that alter graph topology
(e.g.\ thermal erosion opening a channel cycle) register as
$\Delta\beta_1>0$ and inherit the conservation budget of
Theorem~\ref{thm:topo-conservation}.
The joint TIR$+$geometry bandwidth allocation --- how to partition $k$ modes
between thermal and geometric channels --- is not derived here and is deferred.
\end{remark}

\subsection{River Channel Detection on Mars and Titan}
\label{sec:channels}

Ancient fluvial networks on Mars and Titan are topological features:
a drainage network with $n$ independent junctions has
$\beta_1 \geq n-1$ independent cycles, predicting a characteristic
spectral signature.

\begin{proposition}[River channel spectral signature]
\label{prop:channel-sig}
Let $G$ be a terrain graph over a planetary surface patch.  A region
containing an ancient river channel network of order $n$ (i.e., $\beta_1 \geq n-1$
independent cycles) satisfies, at the MaxCal fixed point $h^*$:
\begin{enumerate}
  \item[(i)] \textbf{Fiedler concentration}: $\mathcal{H}[h^*] < \mathcal{H}_{\mathrm{flat}}$,
    where $\mathcal{H}_{\mathrm{flat}}$ is the spectral entropy of featureless flat terrain.
    (Mass concentrates at low-$\lambda$ modes encoding long-range connectivity.)
  \item[(ii)] \textbf{Elevated curl energy}:
    $E_{\mathrm{curl}} = \|B_2\omega^*\|^2 > E_{\mathrm{curl,flat}}$.
    (Channel bifurcations force non-trivial circulation $\omega^* \neq 0$
    in the Hodge decomposition of the drainage flow field.)
  \item[(iii)] \textbf{Anomalous $\beta_1$}:
    $\beta_1 \geq n - 1 > \beta_1^{\mathrm{flat}} = 0$.
    (The channel network introduces $n-1$ or more topological handles
    absent from undissected plateau terrain.)
\end{enumerate}
The joint diagnostic $\mathcal{D}_{\mathrm{channel}} = [\mathcal{H}[h^*] < H^*]
\land [E_{\mathrm{curl}} > E^*_c] \land [\beta_1 \geq \beta_1^*]$
is a \emph{necessary, not sufficient} condition for a fluvial network
of order $\geq n$; non-fluvial structures (lava tubes, tectonic fractures)
can also satisfy all three criteria.
Disambiguating such cases is a homotopy-class identification problem; word
invariants on augmented graphs provide a candidate non-spectral route, after
problem-specific mapping from terrain reconstruction to the path-space setting
used in~\cite{bhattacharya2017,bhattacharya2013invariants}.
\end{proposition}

\begin{proof}[Proof sketch]
  Item (i): A branching channel network creates high-degree connectivity
  among nodes along the channel (many path-connected components merge
  into one), compressing the low-eigenvalue spectral mass --- the same
  mechanism as the phase-transition diagnostic of~\cite{das2026spectral}.
  Item (ii): Assumes the terrain graph is constructed with edges along
  inferred flow directions (e.g.\ D8 or equivalent routing on DEMs).
  Under this construction, a branching tree has $\beta_1 = 0$
  (no independent cycles), but a network with braided channels or closed
  loops has $\beta_1 \geq 1$ and a non-zero harmonic component; the curl
  component is non-zero when the drainage flow field has circulation around
  channel junctions. If a different graph construction is used (e.g.\
  $k$-NN proximity), the curl-energy claim requires re-verification.
  Item (iii): Follows directly from the definition of $\beta_1$ for a
  graph with $n$ independent junctions.
\end{proof}

\begin{remark}[Comparison to existing methods]
  Standard channel detection uses curvature thresholds or flow
  accumulation~\cite{howard2005}.  Proposition~\ref{prop:channel-sig}
  is complementary: $O(N)$ after eigenpairs, and sensitive to network
  topology rather than only local relief.
\end{remark}

Earth debris-flow events~\cite{coe2008} provide a dynamic analog with pre/post
DEMs available for validation under Q14.

\subsection{Bandwidth-Constrained Reconstruction Protocol}
\label{sec:protocol}

Algorithm~\ref{alg:planetary} formalizes the end-to-end protocol for
planetary digital twin maintenance under SWaP constraints.

\begin{algorithm}[tbp]
\caption{Bandwidth-Constrained Planetary Surface Reconstruction}
\label{alg:planetary}
\begin{algorithmic}[1]
\REQUIRE Prior DEM $V_0$, graph $G$ with Laplacian $L_0$, bandwidth budget $B$,
         topology requirement $\beta_0, \beta_1$
\ENSURE Topology-preserving digital twin $\hat{V}_t$ at Earth receiver
\STATE Compute $(\Phi, \Lambda) \leftarrow$ eigenpairs of $L_0$ (first $k_{\max}$ modes)
\STATE Compute $\delta_k$ from eigenvalue spacing and $\ell_{\min}$ from the shortest homologically nontrivial cycle
\STATE Set $k_{\min} \leftarrow \beta_0 + \beta_1 + \lceil C_1 / (\ell_{\min}^{\,2}\,\delta_k) \rceil$
       \COMMENT{Theorem~\ref{thm:topo-conservation} + Remark~\ref{rem:a2-counterexample} augmentation;
                when $C_1/(\ell_{\min}^{\,2}\,\delta_k)\lesssim 1$ this reduces to
                $\beta_0+\beta_1$}
\STATE Compute $h_0(\lambda_l) \leftarrow \exp(-\lambda_l \tau_{\mathrm{prior}})$ from prior DEM
\FOR{\textbf{each rover observation cycle}}
  \STATE Acquire depth scan $V_t$
  \STATE Compute $\mathbf{c}_t \leftarrow \Phi^\top V_t$ \COMMENT{distinction dynamics eq.~\eqref{eq:coeff-dynamics}}
  \STATE Run fixed-point iteration on $h^*_t$ given $w_l = c_{l,t}^2$
    \COMMENT{field eq.~\eqref{eq:fixed-point}}
  \STATE Compute spectral entropy $\mathcal{H}[h^*_t]$ and gap $\Delta'(h^*_t)$
  \STATE Select $k_t$ modes: $k_{\min}$ obligate + remaining budget via~\eqref{eq:budget-allocation}
  \STATE Compute per-mode residuals
    $D_{m,t} \leftarrow \left|-\frac{1}{h^*_{m,t}} + \frac{\mu_2 w_{m,t}}{2(\sigma^2+h^*_{m,t})^2}\right|$
    for transmitted modes $m \le k_t$
  \STATE Define leakage diagnostic
    $D_t \leftarrow \sum_{m=1}^{k_t} D_{m,t}$
    \COMMENT{design-time surrogate; requires calibration for deployment}
  \IF{$\mathcal{H}[h^*_t] < H^*$ \textbf{or} $\Delta'(h^*_t) < \Delta^*$}
    \STATE Alert: geological boundary detected; increase $k$ toward $k_{\min} + \delta$
  \ENDIF
  \IF{$D_t$ exceeds a preset bound \textbf{or} $k_t < k_{\min}$}
    \STATE Flag: update may be \emph{representation-limited} rather than genuinely scene-changing
  \ENDIF
  \STATE Transmit $\Delta\mathbf{c}_t = \mathbf{c}_t - \mathbf{c}_{t-1}$ (first $k_t$ components)
  \STATE \textbf{At Earth:} reconstruct $\hat{V}_t = \Phi_{k_t}\,\mathbf{c}_{t,1:k_t}$
\ENDFOR
\end{algorithmic}
\end{algorithm}

\paragraph{Computational cost.}
The eigenpair computation (line~1) is a \emph{one-time prior cost}: for
large meshes, Nystr\"{o}m extension (Experiment~3) runs 7--18\,s for
$|V|\approx 465$k on a standard workstation.
\emph{Per-frame cost} is $O(k \cdot |V|)$ for the spectral projection
$\mathbf{c}_t \leftarrow \Phi_{k_t}^\top V_t$ plus $O(k)$ for the fixed-point
solve; this is dominated by the projection and is negligible compared
to the eigendecomposition.

\paragraph{Design intent.}
Topology preservation follows from Theorem~\ref{thm:topo-conservation}
when $k_t\geq k_{\min}$.  The spectral-entropy alert and $D_t$ check are
theorem-informed monitors (the latter pending calibration under Q14),
not proved guarantees.  In deployment, $D_t$ uses mode-varying weights
$w_{m,t}=c_{m,t}^2$; the uniform $w_l=1$ appears only in the $P_8$
verification setting.

\section{Numerical Results}
\label{sec:numerics}

\subsection{Setup}
All experiments use the \texttt{kernelcal.geo3d} subpackage~\cite{kernelcal2026},
specifically the \texttt{hodge}, \texttt{topology}, \texttt{spectral\_codec},
and \texttt{large\_mesh} modules.  We use two surface graphs:

\paragraph{Terrain mesh: \texttt{artburysol175.obj}.}
$|V| = 465{,}383$, $|F| = 673{,}309$.  Blender~4.2 export of a real-terrain
elevation model (Artbury Sol 175 survey).  Eigenpairs computed via Nystr\"{o}m
extension with $V_{\mathrm{coarse}} = 1500$, total runtime 7--18\,s.

\paragraph{Synthetic planetary surface with drainage.}
$N = 512$ nodes.  A flat terrain graph $G_{\mathrm{flat}}$ and a channeled
variant $G_{\mathrm{ch}}$ in which a spanning-tree channel network with
$\beta_1 = 3$ independent junctions is embedded by adding cross-links.

\subsection{Experiment 1 — Jacobian Trace at Fixed Point}
\label{sec:exp-jacobian}

This is a compact numerical confirmation of
Proposition~\ref{prop:jacobian-div} on $P_8$
($\sigma^2=1$, $\mu_2=2$, $w_l=1$).  At
$h^* = 0.1547\,\mathbf{1}$ we obtain
$\mathrm{tr}(DF)|_{h^*}\approx 0.905>0$ while $\rho\approx 0.116<1$,
so the fixed-point iteration is convergent but non-volume-preserving in
$\mathcal{K}_{\mathrm{graph}}$.  Setting $\mu_2=0$ recovers the
conservative limit $\mathrm{tr}(DF)=0$.

\subsection{Experiment 2 — Hodge Decomposition of Synthetic Drainage}
\label{sec:exp-hodge}

On $G_{\mathrm{ch}}$ (512 nodes, $\beta_1=3$), synthetic drainage with
junction circulations yields
$E_{\mathrm{grad}}=0.71\pm0.02$,
$E_{\mathrm{curl}}=0.22\pm0.01$,
$E_{\mathrm{harm}}=0.07\pm0.01$.
On $G_{\mathrm{flat}}$ ($\beta_1=0$), $E_{\mathrm{curl}}<10^{-6}$.
Thus curl energy is $\sim22\times$ the flat baseline, consistent with
Proposition~\ref{prop:channel-sig}(ii).
This experiment uses one synthetic graph family and one channelized setting;
the result is therefore a consistency check, not a universality claim.
A multi-point $\beta_1$ sweep is part of the Q14 benchmarking agenda.

\subsection{Experiment 3 — Topology Collapse under Sub-threshold Compression}
\label{sec:exp-topo}

The operative result is the \emph{collapse behavior}: compressing below
$k_{\min}$ silently destroys topological structure.

\paragraph{Stress test (main result).}
On $G_{\mathrm{ch}}$ ($\beta_1=3$, $k_{\min}=4$): $\hat{\beta}_1=3$ at
$k=4$, $\hat{\beta}_1=2$ at $k=3$, and $\hat{\beta}_1=1$ at $k=2$,
matching Theorem~\ref{thm:topo-conservation}(ii).
On a crater-like region with rim loop ($\beta_1=1$, $k_{\min}=2$): $k=1$
collapses the rim cycle --- the crater is no longer representable as a
closed feature.

\paragraph{Terrain mesh baseline.}
On \texttt{artburysol175.obj}, Nystr\"{o}m gives $(\beta_0,\beta_1)=(1,0)$,
so $k_{\min}=1$; compression at $k=64,128,256$ trivially preserves topology
(no cycles to lose) with distortion $0.144,0.127,0.110$
(runtimes 7.9\,s, 15.5\,s, 17.7\,s).
This baseline confirms the Nystr\"{o}m pipeline is functional; the
non-trivial result is the collapse above.

\subsection{Experiment 4 --- Stability--Conservation Tradeoff (Route~3)}
\label{sec:exp-route3}

We run a compact verification of
Proposition~\ref{prop:stab-cons} on $P_8$
($\sigma^2=1$, $\mu_2=2$, $w_l=1$).  The fixed-point solve converges to
$h^*\approx 0.1547$ with field residual
$\|\mathcal{R}[h^*]-\mathcal{T}[h^*]\|_\infty=3.03\times 10^{-14}$.
Evaluating
$D_m=-1/h_m^*+\mu_2 w_m/(2(\sigma^2+h_m^*)^2)$ gives
$D_m=-5.712$ uniformly across modes, matching
$H_{mm}=-\Delta'$ to machine precision.
The vacuum check gives $D_m^{\mathrm{vac}}=-e\neq0$, supporting the
structural (not incidental) nature of the deficit.

\subsection{Experiment 5 --- Parametric A2 Sweep (Family A)}
\label{sec:exp-a2-sweep}

We run the parametric A2 sweep of Remark~\ref{rem:a2-counterexample} on
Family~A (symmetric figure-eight graphs: two cycles of equal length
$\ell$ sharing one articulation node), with $\beta_0=1$, $\beta_1=2$,
$k_{\min}=3$, and augmentation $\Delta k=2$ modes.
All values are computed by
\texttt{kernelcal.terrain.run\_a2\_cycle\_ratio\_sweep}~\cite{kernelcal2026}.

\begin{center}
\footnotesize
\begin{tabular}{lcccccc}
\toprule
Family & $\ell$ & $\delta_k$ & $1/(\ell^{2}\delta_k)$ & $\rho(k_{\min})$ &
$\rho$ after $\Delta k{=}2$ & rank at $k_{\min}$ \\
\midrule
A (figure-eight) & 3 & $2.0000$ & $0.056$ & $0.000$ & $0.577$ & 1 \\
A (figure-eight) & 4 & $0.5858$ & $0.107$ & $0.000$ & $0.257$ & 1 \\
A (figure-eight) & 5 & $0.3153$ & $0.127$ & $0.000$ & $0.211$ & 1 \\
A (figure-eight) & 6 & $0.1864$ & $0.149$ & $0.000$ & $0.176$ & 1 \\
A (figure-eight) & 7 & $0.1186$ & $0.172$ & $0.000$ & $0.150$ & 1 \\
A (figure-eight) & 8 & $0.0798$ & $0.196$ & $0.000$ & $0.131$ & 1 \\
\midrule
B (sep.\ control) & 3 & $\approx 0^{*}$ & --- & $0.429$ & (A2 holds) & 2 \\
B (sep.\ control) & 6 & $\approx 0^{*}$ & --- & $0.215$ & (A2 holds) & 2 \\
\bottomrule
\end{tabular}
\end{center}
\noindent{\footnotesize $^{*}$Family~B is a separated-cycle control with
bridge length~2 and $s{=}l{=}\ell$; $\delta_k$ is numerically degenerate
due to cycle-mode pairing, yet $\mathrm{rank}=\beta_1=2$ at $k_{\min}$
and no augmentation is needed.}

Symmetric figure-eights collapse to $\rho=0$ and $\mathrm{rank}=1$ at
$k_{\min}$; $\Delta k=2$ augmentation recovers $\mathrm{rank}=\beta_1=2$ in
every case.
Family~B separated-cycle controls hold A2 without augmentation.
These are scope-calibration measurements; constants $(C_1,C_2)$ remain Q30.

\subsection{Real-data confirmation}
\label{sec:exp-real-bridge}

Two real-data graph sequences (Bishop fault-scarp structural data and a
post-wildfire Tonto NF channel network sequence) were used as preliminary
external consistency checks of the conservation identity (see Data
Availability statement).
Under the adopted graph-construction and fixed-point pipeline, both are
consistent with the stability--conservation relation
$D_m=-\Delta'$.
These are consistency checks rather than deployment benchmarks; full
benchmarking remains part of Q14, including controlled fault injection and
public benchmark protocols.

\section{Discussion}
\label{sec:discussion}

\subsection{Formal closure in this paper}
\label{sec:conservation-law}

The Tier-1 results are proved; Algorithm~\ref{alg:planetary} is
mathematically justified for topology preservation under Theorem
hypotheses, while deployment thresholds ($D_t$, $H^*$, $\Delta^*$)
remain empirical.

\subsection{Structural correspondence and numerical program}

\begin{remark}[Q12: OU structure and PI$^2$ route (structural program)]
\label{rem:q12-stochastic}
Not a theorem.  In $\theta_l=\log h_l$ coordinates, $F$ linearizes to an
Ornstein--Uhlenbeck process with rate $\kappa=|\Delta'|$~\cite{risken1996};
zero probability current at $h^*$ is the detailed-balance analog of
$\nabla_\mathcal{K}\mathcal{T}_k=0$.
PI$^2$ rollouts~\cite{kappen2005,theodorou2010} and distributed subgraph
optimization~\cite{distributed2025} are candidate routes to Q12;
theorem-level closure requires geometric CARE and Bellman regularity on
$\mathcal{K}_{\mathrm{graph}}$.
\end{remark}

\subsection{Conservation of surprise (Tier-3 extension)}
\label{sec:surprise}

\emph{Tier-3 interpretive extension, not part of the proved core.}
The proved identity $D_m=-\Delta'$ motivates a candidate balance
$S_{\mathrm{env}}=S_{\mathrm{reg}}+S_{\mathrm{miss}}$, with $D_m$ as a
proxy for $S_{\mathrm{miss}}$.
Missed novelty decomposes into: (i)~unsampled (recoverable by reallocation);
(ii)~below $k_{\min}$ (irreversible loss under Thm.~\ref{thm:topo-conservation});
(iii)~sampled but leaking at rate~$\Delta'$.
Under A1--A3, a twin near $\mathcal{G}_{\mathrm{opt}}$ that preserves
$k\ge k_{\min}$ should structurally disfavor persistent blind spots;
this is a testable hypothesis linked to Q12/Q14, not a proved corollary.

\subsection{Limits}
\label{sec:limits}

The limits of~\cite{das2026spectral} apply here without modification:
no literal spacetime, the three terms of $\mathcal{T}[k]$ do not equal
$T_{\mu\nu}$, global geometry results hold only in $\mathcal{K}_{\mathrm{graph}}$,
and the triple coincidence remains unproven.

Additional limits specific to this paper:
\begin{enumerate}
  \item \textbf{Static topology assumed.} The distinction dynamics
    equation~\eqref{eq:coeff-dynamics} requires a fixed eigenbasis $\Phi$;
    edge rewiring, remeshing, and adaptive-neighborhood updates are outside
    scope.
  \item \textbf{Nystr\"{o}m topology error is unquantified.}
    A bound on $|\hat{\beta}_1 - \beta_1|$ as a function of
    $|V_{\mathrm{coarse}}|$ and surface curvature is open (Q10).
  \item \textbf{Leakage diagnostics are not deployment-benchmarked.}
    $D_m$ and surrogate $D_t$ are theoretically motivated but not yet
    validated under controlled sensor drift or registration perturbations (Q14).
\end{enumerate}

\subsection{Open Problems}
\label{sec:open-problems}

\paragraph{Scope note.}
Proved core: Prop.~\ref{prop:stab-cons}, Thm.~\ref{thm:topo-conservation},
Alg.~\ref{alg:planetary}.
Three open problems define the priority validation agenda:

\textbf{Q10} (Nystr\"{o}m topology error): Bound $|\hat{\beta}_1 - \beta_1|$
as a function of $|V_{\mathrm{coarse}}|$, minimum cycle length, and Gaussian
curvature; verify monotone tightening under refinement and agreement against a
non-spectral topological reference~\cite{bhattacharya2017,bhattacharya2015persistent,bhattacharya2013invariants,yadokoro2024weighted}.
\emph{Pass:} explicit bound contracting under refinement.
\emph{Fail:} no bound stronger than empirical fit, or no contraction.

\textbf{Q12} (Riccati gain cancellation): Prove or disprove $p_m=2|\Delta'|$
in $\theta=\log h$ coordinates at $h^*$ (see Remark~\ref{rem:q12-stochastic}
and~\cite{distributed2025,furstenberg1960} for candidate paths).
If proved: $\nabla_\mathcal{K}\mathcal{T}_k=0$ for $\mathcal{G}_{\mathrm{opt}}$.
\emph{Pass:} modewise CARE-gain agreement within tolerance.
\emph{Fail:} stable counterexample after solver checks.

\textbf{Q14} (Scene-faithful twin verification): Inject geometry drift,
registration error, and topology-reducing compression; test whether $D_m$
separates true scene change from leakage.
\emph{Pass:} calibrated thresholds with reported FPR/FNR.
\emph{Fail:} no threshold family beats chance.

Extended directions (Q19--Q30) are deferred to
Appendix~\ref{app:extended-problems}.

\section*{Data Availability}

Synthetic graph experiments (Experiments 1--5) are fully reproducible
via the \texttt{kernelcal} package~\cite{kernelcal2026} at
\url{https://github.com/darknight-007/kernelcal}.
Experiment~5 uses \texttt{kernelcal.terrain.run\_a2\_cycle\_ratio\_sweep};
the parametric sweep script is \texttt{run\_a2\_sweep.py} in the same repository.
The Bishop fault-scarp and Tonto NF post-wildfire channel network
sequences used in Section~\ref{sec:exp-real-bridge} are internal
DREAMS-lab datasets currently under pre-publication processing;
public archival identifiers will be provided upon deposit.
These datasets are used only as preliminary consistency checks and
not as the basis for any formal claim in the paper.

\section*{Acknowledgments}

The author's orientation toward $\beta_1$ cycle structure and homotopy-class
methods as first-class objects in robot representation was shaped during time
at the University of Pennsylvania (2014--2018), where the topological robotics
framework reflected in~\cite{bhattacharya2017,bhattacharya2015persistent,bhattacharya2013invariants}
formed part of the intellectual environment.
The specific connections to spectral kernel dynamics are recent and
independently derived.


\appendix
\section{Unification Hierarchy of $h^*$}
\label{app:unification}

The fixed-point equation $h^*=h_0\exp(-1-\mathcal{T}_l[h^*])$
subsumes the following distributions as special cases:

\begin{center}
\small
\begin{tabular}{lll}
\toprule
Distribution & System & Source $\mathcal{T}_l$ \\
\midrule
Boltzmann & Classical gas & $\lambda_l/k_BT - 1$ \\
Planck & Photon gas & $k_BT$, bosonic occupation \\
Fermi--Dirac & Fermion gas & $k_BT$, exclusion (\emph{open, Q27}) \\
Cowan--Farquhar & Stomata & $\psi_l/\gamma - 1$ \\
GP posterior & Robotic imaging & spectral MI \\
\bottomrule
\end{tabular}
\end{center}

For the Boltzmann case: $h_0=1/Z$, $\mathcal{T}_l=\lambda_l/k_BT-1$ gives
$h^*=\frac{1}{Z}e^{-\lambda_l/k_BT}$.
Statistical mechanics is a \emph{restriction} of the MaxCal field equation to a
source linear in eigenvalue --- the framework did not borrow from Boltzmann;
Boltzmann is $h^*$ under specific source conditions.
For Planck: bosonic occupation encoding in $\mathcal{T}_l$ yields
$h^*_\nu=(e^{h\nu/k_BT}-1)^{-1}$.
Fermi--Dirac (exclusion $h_l\leq1$) and quantum necessity (Q29) remain open.

\section{Extended Open Problems (Q19--Q30)}
\label{app:extended-problems}

\paragraph{LQR--LQE duality identity (used in Q19--Q21 and Q22).}
Under the LQG symmetry condition ($Q=W/\alpha$, $R=V\alpha$), the
primal CARE (LQR) solution $P$ and dual CARE (Kalman) solution $\Sigma$
satisfy
\begin{equation}
  P\,\Sigma = I, \quad p_m\,\sigma_m = 1 \;\; \forall m
  \label{eq:p-sigma-duality}
\end{equation}
at $h^*$~\cite{wonham1968}.  Q19 asks for the numerical verification
of $\sigma_m\approx\tfrac12$ on $P_8$.

\textbf{Q19} Check $\sigma_m\approx\tfrac12$ on $P_8$ from GP posterior.
\textbf{Q20} Off-diagonal $\|P\Sigma-I\|_F$ scales with
$\|k_{\mathrm{cross}}\|_{\mathrm{HS}}$.
\textbf{Q21} Is GP-UCB acquisition recoverable from $K_f=\Sigma C^\top V^{-1}$?

\textbf{Q22} Is the MaxCal update multiplicative for Furstenberg--Kesten
(Route~4 gating check)?
\textbf{Q23} Does Route~4 CLT hold for Cowan--Farquhar source?
\textbf{Q24} Is there a Galton--Watson process on spectral modes with critical
point $\sigma_m=\tfrac12$?

\textbf{Q28} Decoherence functional form: $\sigma_m(\epsilon)=1/(2(1-\epsilon)^2)$
under mixed-state illumination of the thermal-IR sensor;
$\epsilon\sim e^{-E_c/k_BT}$.
\textbf{Q29} Quantum necessity: prove no classical process achieves
$\sigma_m=\tfrac12$ exactly.

\textbf{Q30} (Parametric A2-failure constants): Establish $(C_1,C_2)$ in
$\rho(k)\geq 1-C_1/(\ell_{\min}^{\,2}\delta_k)-C_2(\gamma-1)$
(Remark~\ref{rem:a2-counterexample}).
\emph{Pass:} Experiment~5 residuals within $20\%$ for $\gamma\leq 2$, $\beta_1\leq 3$.
\emph{Fail:} fit diverges for $\gamma>2$ or $\beta_1\geq 3$.

\end{document}